\documentclass[11pt,oneside]{amsart}

\usepackage{amsmath}

\usepackage[numbers,sort]{natbib}

\usepackage{amsthm,amssymb,mathrsfs,mathtools}
\usepackage[dvipsnames,x11names,svgnames]{xcolor}

\usepackage[colorlinks=true,linkcolor=NavyBlue,citecolor=DarkGreen, urlcolor=blue]{hyperref}

\usepackage{geometry}

\renewcommand{\r}{\rho}

\newcommand{\g}{\gamma}
\newcommand{\s}{\sigma}

\newcommand{\be}{\begin{equation}}
\newcommand{\ee}{\end{equation}}
\newtheorem{theorem}{Theorem}[section]

\newtheorem{corollary}[theorem]{Corollary}

\newtheorem*{hyp*}{Hypothesis}
\newtheorem*{hypA*}{The Alternative Hypothesis}
\newtheorem*{con*}{Conjecture}
\newtheorem*{cor*}{Corollary}
\theoremstyle{definition}

\theoremstyle{remark}

\begin{document}
 
\title[Consecutive moderate  gaps between zeta zeros]{Consecutive moderate gaps between zeros of the Riemann zeta function}

\author{Steven M. Gonek}
\address{Department of Mathematics, University of Rochester, Rochester, NY 14627, USA}
\email{\url{steven.gonek@rochester.edu}}

\author{Anurag Sahay}
\address{Department of Mathematics, Duke University, 120 Science Drive, Durham NC 27708, USA}
\email{\url{anurag.sahay@duke.edu}}

\begin{abstract}
Let $0<\g_1\leq \g_2 \leq   \cdots $ denote the ordinates of nontrivial zeros of the Riemann zeta function with positive imaginary parts.
 For $c>0$ fixed (but possibly small),   $T$ large, and $\g_n\leq T$, we call a gap $\g_{n+1}-\g_n$ between consecutive ordinates ``moderate''  if $\g_{n+1}-\g_n \geq 2\pi c/\log T$. 
We investigate whether infinitely often  there exists  $r$ consecutive moderate gaps between ordinates $\g_{n+1}-\g_n, \g_{n+2}-\g_{n+1}, \ldots , \g_{n+r}- \g_{n+r-1}$. 
\end{abstract}
\maketitle

\section{Introduction}\label{Sec Intro}
Our present knowledge of the statistics of the nontrivial zeros of the Riemann zeta function, $\zeta(s)$, has a curious feature. On the one hand, we have conjectures that allow us to  provisionally provide precise answers to a wide variety of  questions about the zeros. For example,  we expect all the nontrivial zeros to be on the critical line $\Re s = 1/2$ (the Riemann Hypothesis). As another example,  we expect the  $n$-level correlations of zeros high up on the critical line to be the same as the  $n$-level correlations of the eigenvalues of $N\times N$ random matrices  from the gaussian unitary ensemble (with $N$ depending on how high up the critical line we are). This is the so-called GUE Hypothesis or Montgomery--Odlyzko law. Thus, a reasonable conjecture for many questions about the statistics of the zeros of $\zeta(s)$ can often be found by working out the analogue for GUE matrices. 

On the other hand, rigorous results about the zeros generally fall far short of our conjectural predictions, even under the assumption of the Riemann Hypothesis. For example, the answer to the relatively straightforward question ``are there infinitely many  gaps between consecutive zeros of the zeta function that are less than half  the average gap size?" is not known. Indeed, if it were answered affirmatively in a  strong enough quantitative form, it would have deep implications for arithmetic by eliminating the hypothetical Siegel zero for quadratic Dirichlet $L$-functions \cite{CI}. 

The goal of this  note is to discuss another  natural, simple-looking problem about zero statistics, where the gap between what we expect to be true and what we can prove is  large. To state it, we
number the nontrivial zeros of the zeta function with positive imaginary parts as
$ \rho_n = \beta_n + i \gamma_n, $
with $0<\g_1 \leq \g_2\leq \cdots$, where if two or more zeros have the same ordinate, we list them arbitrarily (among themselves). 
As is well known,
\be\label{zero ct}
N(T) =\sum_{0<\g_n\leq T} 1=\frac{T}{2\pi} \log \frac{T}{2\pi}-\frac{T}{2\pi} +O(\log T).
\ee
It follows from this  that the average gap between two consecutive ordinates of zeros in the interval $[0, T]$ is $2\pi/\log T$,
and that there are positive constants $C$ and $c$ such that every interval $[T, T+C]$ with $T$ large contains a pair of consecutive ordinates $\g_n, \g_{n+1}$ with $\g_{n+1}-\g_n\geq   2\pi c/\log T$. 
We refer to such a gap  as a ``moderate gap''.  
This observation is classical and has frequently been used  to show that one can choose a sequence of heights $T\to\infty$ between two consecutive ordinates $\g_n <\g_{n+1}$ such that  $\min(\g_{n+1} -T , T-\g_n)\gg 1/\log T$.  This then allows one to show, for example, that on the segment
$-1\leq \s\leq 2$ we have  $\zeta'/\zeta(\s+iT)\ll (\log T)^2$ (see the beginning of \S\ref{Sec: 3}). 

The question we wish to investigate is the following:\
how frequently are there $r+1$ consecutive ordinates $\g_n, \g_{n+1}, \ldots, \g_{n+r} \in(0, T]$ such that the $r$ consecutive gaps 
$\g_{n+1}-\g_n, \g_{n+2}-\g_{n+1}, \ldots, \g_{n+r}-\g_{n+r-1} $ are all greater than or equal to  $2\pi c/\log T$ for some fixed positive $c$?

Notice that the problem of consecutive small gaps, that is, gaps that are each less than    $2\pi c/\log T$ in size,  would to some extent be subsumed by strong upper bounds for  $r$-gaps  between zeros, that is, for the differences $\gamma_{n+r} - \gamma_n$ (see, for example, \cite{CT-B} and references therein). For example, if we knew that $\gamma_{n+2} - \gamma_n \leq 2\pi c/\log T$, then $\gamma_{n+1} - \gamma_{n}$ and $\gamma_{n+2} - \gamma_{n+1}$ would each be $<2\pi  c/\log T$.
The complementary question about consecutive moderate gaps is not so easily reduced. For instance,  if $\gamma_{n+2} - \gamma_{n}$ is large, it could still be the case that one of $\gamma_{n+2} - \gamma_{n+1}$ or $\gamma_{n+1} - \gamma_{n}$ is small. 

Most of our results assume widely believed hypotheses about the zeros of $\zeta(s)$ such as the Riemann Hypothesis or the Pair Correlation Conjecture. Indeed, we are unable to establish unconditionally that there are infinitely often two consecutive  gaps of size $\geq 2\pi c/\log T$ for any $c>0$. In fact, for all we know unconditionally, every gap of size  $\geq 2\pi c/\log T$ might be followed by a gap of size $<2\pi c/\log T$, or even by a gap of size zero (coming from multiple zeros). 

In the next section we derive answers to our  main question based on several different hypotheses. 
In \S\ref{Sec: 3} we prove an unconditional result related to our question and in \S\ref{Sec: 4} we present an application of our results.

\subsection*{Recent results}
Subsequent to the initial dissemination of our manuscript online, two works \cite{Lamzouri, ARTZ} have recently appeared which unconditionally address our main question for small values of $r$. In particular, Lamzouri \cite{Lamzouri} achieves a positive proportion of gaps for $r\leq 5$ and $c$ sufficiently small. These papers build on ideas derived from \cite{Anthropic}. 

\section{Conditional results} \label{Sec: 2}

\subsection{The Alternative Hypothesis}
One of the strongest well known hypotheses we can bring to bear on our question is also one of the least believable, namely the Alternative Hypothesis. Several forms appear in the literature. The one we use is due to S. Baluyot~\cite{Baluyot} (see \cite{FGL} for a stronger form and see \cite{RV} and references therein for a generalization of AH to higher correlations).
\begin{hyp*}[Alternative Hypothesis]
For each nonnegative integer $n$ there is an integer $k_n$ such that
\[
\frac{\g_{n+1}}{2\pi}\log \g_{n+1} - \frac{\g_{n}}{2\pi} \log \g_{n} = \frac12  k_n +O(|\g_{n+1}-\g_{n}| \ \psi(\g_{n})),
\]
where $\psi(\g)$ is a function such that $\psi(\g)\to \infty$ and $\psi(\g)=o(\log \g)$ as $\g\to\infty$.
\end{hyp*}
 
Let $B_{k/2}(T)$ be the set of zeros $0<\g_n\leq T$ such that $k/2$ is closest among all half-integers to 
$  (\g_{n+1}\log \g_{n+1})/2\pi -  (\g_{n} \log \g_{n})/2\pi$ with the convention that
\[ 
\max \Big\{\frac{k}{2}-\frac 14, 0 \Big\} \leq \frac{\g_{n+1}\log \g_{n+1}}{2\pi } -  \frac{\g_{n} \log \g_{n}}{2\pi  }< \frac{k}{2}+\frac 14.
\]
Also define
\[
p_{k/2}(T) =\frac{1}{N(T)}\sum_{\g_n\in B_{k/2}(T)} 1.
\]
S. Baluyot~\cite{Baluyot} has shown that if Riemann Hypothesis and the Alternative Hypothesis are true and all zeros of $\zeta(s)$ are simple, then as
$T\to\infty$, 
\[ p_0(T)=o(1).
\]
From this it  follows that 
\[
\frac{\g_{n+1}}{2\pi}\log \g_{n+1} - \frac{\g_{n}}{2\pi} \log \g_{n} \geq \frac14   +o(1)
\]
for   all but $o(N(T))$ of the  $\g_n\in(0, T].$ Consequently, we have
\[
 \g_{n+1} -  \g_{n} \geq \frac{ \pi}{2\log T} (1+o(1))   
\]
for  all but $o(N(T))$ of the  $\g_n\in(0, T]$. Thus, for any fixed integer $r\geq 1$, any $\epsilon>0$, and for $T$ sufficiently large, we can find   
$(r+1)$-tuples  of ordinates $(\g_n, \g_{n+1}, \ldots, \g_{n+r})$ with $\g_n \in(0, T]$ such that the $r$ consecutive gaps 
$\g_{n+1}-\g_n, \g_{n+2}-\g_{n+1}, \ldots, \g_{n+r}-\g_{n+r-1} $ are each greater than or equal to  $(\pi/2-\epsilon)/\log T $.

\subsection{The GUE Hypothesis} \label{sec: GUE}

We now discuss what is likely to be the ultimate truth concerning our question, namely what the GUE Hypothesis \cite{M, O} predicts. We first recall the hypothesis.

\begin{hyp*}[GUE Hypothesis]
For each integer $n \geq 1$, the $n$-level correlations of the sequence $\{\gamma_m \frac{\log T}{2\pi} : \gamma_m \leq T\}$ converge as $T \to \infty$ to the limiting $n$-level correlations of the random matrices in the GUE ensemble in the large dimension limit.
\end{hyp*}

We refer the reader to \cite{RS} for a more precise statement of the hypothesis. The limiting $n$-level correlations for GUE are known to be described by the \emph{sine-kernel} point process. The reader may consult \cite{AR} for a description of the conjecture along these lines.

Before discussing   $r$ consecutive gaps between zeros, we first consider the case $r = 1$. That is, 
 we want to know the proportion of   $\gamma_n \in (0,T]$,
 for which  $\gamma_{n+1} - \gamma_n \geq 2\pi c/\log T$. This is tantamount to understanding
\begin{equation} \label{prob} \lim_{T \to \infty} \frac{1}{N(T)}\#\{ \gamma_n \in (0,T] : \gamma_{n+1} - \gamma_n \leq 2\pi c/\log T \}. \end{equation}

The quantity \eqref{prob} is well-studied in the literature under the name \emph{nearest neighbor spacing},  (see, for e.g., \cite{O}). Under the GUE Hypothesis, a concrete representation for \eqref{prob} can be given in terms of a Fredholm determinant, which takes the shape
\[ \int_0^c p_2(0,t) dt \]
where $p_2(0,t) = \frac{d^2}{dt^2} \det(I-Q_t)$, and  where $Q_t$ is a particular trace-class operator on $L^2(-1,1)$. (The subscript in $p_2(0,t)$ simply refers to GUE.) 
If the $0$ here is replaced by $\ell-1$ for $\ell \geq 1$, then a similar (but more cumbersome) formula is obtained for 
\begin{equation} \label{prob2} \lim_{T \to \infty} \frac{1}{N(T)}\#\{ \gamma_n \in (0,T] : \gamma_{n+\ell} - \gamma_n \leq 2\pi c_\ell/\log T \}. \end{equation}
For details, see \cite{Me} or \cite{dCM}. 
 Computing the quantities \eqref{prob2} for $1 \leq \ell \leq r$ is essentially understanding the individual probability distributions of $r$ random variables, namely the normalized limiting distributions of $\gamma_{n+\ell} - \gamma_n$. To understand the quantity of interest to us, however, we would need to know the \emph{joint} distribution of these random variables, that is, the quantity

\begin{equation} \label{jprob} \lim_{T \to \infty} \frac{1}{N(T)}\#\{ \gamma_n \in (0,T] : \gamma_{n+\ell} - \gamma_n \leq 2\pi c_\ell/\log T \text{ for every } 1 \leq \ell \leq r \}. \end{equation}

It would be interesting  to describe \eqref{jprob} using Fredholm deteminants, but   we will not pursue this here. An adequate understanding of \eqref{jprob} should show that for every $c > 0$, a positive proportion of zeros have $r$ consecutive moderate gaps of size $\geq 2\pi c/\log T$ after them. 

\subsection{The Pair Correlation Conjecture} \label{sec: pcc}

We first recall  Montgomery's Pair Correlation Conjecture~\cite{M}.

\begin{con*}[Pair Correlation Conjecture]
Let
\be\label{f alpha}
 f(\alpha) = \int_0^\alpha\bigg(1 - \Big(\frac{\sin\pi u}{\pi u}\Big)^2 \bigg) du. 
 \ee
Then for $c>0$ we have
\[ \frac{1}{N(T)}\sum_{\substack{0 <\gamma_m,\gamma_n \leq T\\0<\gamma_n - \gamma_m < 2\pi c/\log T} } 1 = f(c)+o(1). \]
\end{con*}
 
 This conjecture follows from the GUE hypothesis, namely, it is the special case of looking only at the $2$-level correlations.
 
  Define
\[
g_{n+j} =\g_{n+j+1}-\g_{n+j}\qquad \quad (j=0, 1, 2, \ldots) 
\]
and set
\[
N_r(T, c)=\sum_{\substack{  0<\g_n\leq T \\ g_n,\ g_{n+1},\ \ldots,\ g_{n+r-1} \geq 2\pi c/\log T }} 1.
\]
For $j= 1, 2, \ldots$ let
\[
S_j(T,c) = \Big\{  0<\g_n\leq T : g_n,\ g_{n+1},\ \ldots,\ g_{n+j-1} \geq \frac{2\pi c}{\log T}\;\hbox{and}\; g_{n+j}<    \frac{ 2\pi c}{\log T} \Big\}.
\]
Then  the sets $S_j(T,c)$ and $S_k(T,c)$ are disjoint when $j\neq k$, and for $r\geq 1$ we clearly have
\be\label{N_r}
N_r(T, c)=N(T) -\sum_{j=1}^r |S_j(T, c)|. 
\ee

\begin{theorem} \label{Thm: pcc}
Assume the Riemann Hypothesis and the Pair Correlation Conjecture. Then for any fixed integer $r \geq 1$ and $c > 0$, we have 
\be\label{Nr lwr bd 1}
N_r(T, c) \geq (1 - rf(c) + o(1)) \ N(T) 
 \ee
as $T \to \infty$.
\end{theorem}

\begin{proof} 
If $r$ is fixed, $\gamma_n \leq T$, and  $T$ is sufficiently large, then  $\gamma_{n}, \ldots, \gamma_{n+r} \leq T+1$. This follows from $N(T+1) - N(T)  \sim \log T/2\pi$ under RH by using Littlewood's bound for $S(t)$. Thus, we see that
\begin{equation} \notag
|S_j(T, c)| \leq \sum_{\substack{0 \leq \gamma_n \leq T+1 \\ 0<\gamma_{n+1} - \gamma_n < \frac{2\pi c}{\log T}}} 1 + \sum_{\substack{0 \leq \gamma_n \leq T+1 \\ \gamma_{n+1} = \gamma_n}} 1.
 \end{equation}
This is an overcount since we are ignoring the restrictions that the definition of $S_j(T, c)$ places on $g_{n}, g_{n+1}, \ldots, g_{n+j-1}$. 
Now the Pair Correlation Conjecture implies that almost all zeros (in the sense of density) are simple (see \cite[p.~184]{M}), so the second sum on the right-hand side  is $o(N(T))$.
Thus, 
\begin{equation} \notag
|S_j(T, c)| \leq \sum_{\substack{0 \leq \g_m, \gamma_n \leq T+1 \\ 0<\gamma_{n } - \gamma_m < \frac{2\pi c}{\log T}}} 1 + o(N(T)) .
\end{equation}
 By the Pair Correlation Conjecture again, we see that
\begin{equation} \notag
|S_j(T, c)| \leq  (f(c)+ o(1) ) N(T).
\end{equation}
Combining this with \eqref{N_r}, we obtain the theorem.
\end{proof}

\begin{corollary}
Assume the Riemann Hypothesis and the Pair Correlation Conjecture. Then for any fixed integer $r \geq 1$, there exist numbers $c > 0$  for which a positive proportion of the $0<\g_n\leq T$ give rise to $r$-tuples of consecutive gaps $g_n, g_{n+1}, \ldots , g_{n+r-1}$, 
each of length $\geq 2\pi c/\log T$. Moreover, this holds any $c \leq 1/\pi$ when $1 \leq r \leq 28$ and for any $c < (3/\pi)^{2/3}r^{-1/3}$ when $r>28$. 
\end{corollary}

\begin{proof}
The first assertion follows from \eqref{Nr lwr bd 1} on noting that, by \eqref{f alpha}, $f(\alpha)\to 0$ as $\alpha \to 0^+$. 
To deduce  the  second assertion, we integrate the inequality 
\[
1 - \Big(\frac{\sin\pi u}{\pi u}\Big)^2 \leq \frac{(\pi u)^2}{3},
\]
valid for $u\leq 1/\pi$,
over $[0, c]$. This gives the upper bound 
\begin{equation} \label{eqn:PCCimpliesSH} f(c)\leq (\pi/3)^2 c^3 \text{ for any } 0\leq c\leq 1/\pi. \end{equation}
Inserting this into \eqref{Nr lwr bd 1}, we find that $N_r(T,c)>0$ when $c< (3/\pi)^{4/3}r^{-1/3}$, $c\leq 1/\pi$, and $T$ is sufficiently large.
\end{proof}

Here is a table of values $c_r$ for various values of $r$, where $c_r$ is defined as the unique positive root of the equation $1-r f(c_r)=0$. Thus, if $0<c<c_r$, a positive proportion of the zeros 
$0<\g_n\leq T$ are initial zeros of $r$ consecutive gaps of length $\geq 2\pi c/\log T$, once $T$ is sufficiently large. 
 
\hskip2in
\begin{tabular}{|l | c | }
\hline
$r$& $c_r$ \\ \hline 
$1$ &$1.46389$ \\ \hline 
$2$ & $0.951371$ \\ \hline
$3$ & $0.780111$ \\ \hline
$4$ & $0.68697$ \\ \hline
$5$ & $0.625737$ \\ \hline
$6$ & $0.581289$ \\ \hline
$7$ & $0.546994$ \\ \hline
$8$ & $0.519411$ \\ \hline
$9$ & $0.496551$ \\ \hline
$10$ & $0.477168$ \\ \hline
$20$ & $0.370163$ \\ \hline
$100$ & $0.21138$ \\ \hline
$1000$ & $0.0972135$ \\ \hline
\end{tabular}

%
%

\vskip.3in


\subsection{The Well-Spacing Hypothesis}

In this section we investigate what can be shown under the following Well-Spacing Hypothesis, which is weaker than Montgomery's Pair Correlation Conjecture. The argument from \S\ref{sec: pcc} can be adapted to give a positive proportion of $r$-tuples of consecutive moderate gaps under this hypothesis as well, albeit with smaller values of $c$. The same Well-Spacing Hypothesis was assumed in \cite{Cicek,CG} and similar hypotheses were considered in \cite{BH87,Hejhal,Selberg,BH,Lester,Kirila,CGK}.

\begin{hyp*}[Well-Spacing Hypothesis] \label{hypothesis}
There exists a real number $\delta>0$ and a positive constant $M$ depending on $\delta$ such that,  uniformly for  $0<u<1$, we have
\[
 \limsup_{T\to\infty}\frac{1}{N(T)}\Big| \Big\{0 < \g_n \leq T:0\leq \g_{n+1}-\g_n <\frac{2\pi u}{\log T}\Big\}\Big| \leq M u^\delta .
\]
\end{hyp*}

The Well-Spacing Hypothesis with $\delta=3$ follows from Montgomery's Pair Correlation Conjecture; indeed, it follows easily from \eqref{eqn:PCCimpliesSH}. Notice that Well-Spacing for a given $\delta$ in turn implies it for every smaller positive $\delta$. Note also 
that the Well-Spacing Hypothesis implies that all but $o(N(T))$ of the zeros are simple.

The method used to prove Theorem~\ref{Thm: pcc} may easily be modified to prove the following result.
\begin{theorem} \label{Thm: well-spacing}
Assume the Riemann Hypothesis and that the Well-Spacing Hypothesis holds for $\delta$. Then for $c > 0$ and any fixed integer $r \geq 1$ we have 
\be\notag
N_r(T, c) \geq (1 - rMc^\delta + o(1)) \ N(T) .
 \ee
as $T \to \infty$.
\end{theorem}

It immediately follows from this that under the hypotheses of the theorem,  if $c>0$ is sufficiently small as a function of $r$ and $\delta$, then 
a positive proportion of the ordinates $0<\g_n\leq T$ give rise to $r$-tuples of consecutive gaps $g_n, g_{n+1}, \ldots , g_{n+r-1}$, 
each of length $\geq 2\pi c/\log T$.

\section{An unconditional result} \label{Sec: 3}

If we do not assume any hypothesis about the zeros, it could be the case that for every $c>0$, every  vertical gap between consecutive zeros of size $\geq 2\pi c/\log T$ is followed  by a gap of size $<2\pi c/\log T$. In this section we show that even if this is the case, there are still gaps of moderate size (and even $r$ of them) that do not have many ordinates of zeros between them.
 
\begin{theorem}\label{thm: uncond}
Let $r$ be a positive integer, $T\geq 2$,   and let $\epsilon>0$. Suppose that $m$ is an integer divisible by 10 and that $h=2\pi m/\log (T/2\pi)$. If $m$ is sufficiently large and fixed, then for all $T$ sufficiently large, there exists a set $\mathscr T$ of measure $>(1-\epsilon) T$ in $(T, 2T]$ such that for every $t\in \mathscr T$, each of the $r$
intervals $I_j(t) = \big(t+(j-1)h), t+jh \big],\, j=1, 2, \ldots, r,$ contains $<\frac32 m$ ordinates of zeros and each $I_j(t)$ contains at least one consecutive pair of zeros with a gap of length
$\geq 16\pi/(15 \log T)$.
\end{theorem}

\begin{proof}
Let 
\[
S(t)=\frac1\pi \arg\zeta(\tfrac12+it),
\]
where, if  $t$ is not the ordinate of a zero, the argument is obtained by continuous variation along the  lines connecting $2, 2+it$ and  $\frac12+it$, 
beginning with the value $S(2)=0$. If  $t$ is the ordinate of a zero, we let $S(t)=\lim_{\epsilon\to 0^+}S(t+\epsilon)$.
By Theorem 9.3 of Titchmarsh \cite{T}, 
\be\notag
N(t) =\frac{t}{2\pi} \log \frac{t}{2\pi}-\frac{t}{2\pi} +\frac78 +S(t)+O\Big(\frac{1}{t}\Big)
\ee
as $t\to\infty$. From this we find that, if $T<t\leq 2T$ and $h\ll 1$,   then
\be\label{S-S}
S(t+h)-S(t) = N(t+h)-N(t) -\frac{h}{2\pi} \log \frac{t}{2\pi}+O\Big(\frac{1}{t}\Big).
\ee
Fujii~\cite{F} has shown that 
\be\notag%
\int_{T}^{2T} (S(t+h)-S(t))^2 dt \ll T \log(3+h\log T).
\ee
Thus, by \eqref{S-S},
\be\notag%
\int_{T}^{2T} \Big(N(t+h)-N(t) -\frac{h}{2\pi}\log \frac{T}{2\pi}\Big)^2 dt \ll T \log(3+h\log T).
\ee
We now set $h=2\pi m/\log (T/2\pi)$ where $m$ is a positive  integer. Then 
\be\notag%
\int_{T}^{2T} (N(t+h)-N(t) -m)^2 dt \ll T \log 2m.
\ee
From this we see that the measure of the set of $t\in(T, 2T]$ for which 
\[
|N(t+h)-N(t) -m| \geq \frac m2
\]
is $O(T(\log 2m)/m^2)$. Replacing $m$ by $m/10$, we see that the measure of the set $t \in (T,2T]$ for which
\[
|N(t+h/10)-N(t) - m/10| \geq \frac m{20}
\]
is also $O(T(\log 2m)/m^2)$. Thus, for $t \in (T, 2T]$, except for a set of measure $O(T\log 2m/m^2)$, the following inequalities hold:
\begin{equation} \label{ineq 1} \begin{gathered} 
\frac m2 < N(t+h)-N(t) <\frac{3m}{2}, \\
\frac m{20} < N(t+h/10)-N(t) <\frac{3m}{20}, \\
\frac m{20} < N(t+h)-N(t+9h/10) <\frac{3m}{20}.
\end{gathered}
\end{equation}

Suppose now that we fix an integer $r$ and let $\epsilon>0$ be small.
Then if $m$ is sufficiently large, \eqref{ineq 1} will hold for all $t\in (T, 2T]$ except for a set of measure $\leq \epsilon T/ r$.
Thus, for such $m$  there exists a set $\mathscr T$ of $t\in (T, 2T]$ of   measure greater than $(1-\epsilon) T$ for which
we have

\begin{equation*} \begin{gathered} 
\frac m2 < N(t+jh)  -N(t+(j-1)h) <\frac{3m}{2}, \\
\frac m{20} <  N(t+(j-1)h + h/10) -N(t+(j-1)h) <\frac{3m}{20}, \\
\frac m{20} <  N(t+jh) -N(t+jh - h/10) <\frac{3m}{20}.
\end{gathered}
\end{equation*}
simultaneously  for $j=1, 2, \ldots , r$. For $t\in \mathscr T$, let $I_j(t)=(t+(j-1)h), t+jh],\, j=1, 2, \ldots, r$. The first equation above thus says that $I_j(t)$ contains many zeros, and the latter two ensure that the first tenth and the last tenth of $I_j(t)$ contain at least one zero. Let $M_1, M_2, \ldots, M_r$, denote  the number of ordinates of  zeta zeros in these $r$ intervals, so that
$m/2<M_j<3m/2$. The   average spacing between the  $M_j$ ordinates in $I_j(t)$ is then at least
\[
\frac{8h/10}{M_j-1} >\frac{8\pi m}{5 M_j \log (T/2\pi)} > \frac{16\pi}{15\log T}.
\]
Thus, among the $M_j$ ordinates in $I_j(t)$, there is a gap  between at least one pair of consecutive ordinates of size $\geq 16\pi/(15\log T)$. In other words, given $r$ and $\epsilon>0$, if the integer $m$ is large enough, then for $T$ sufficiently large, there exists a set $\mathscr T$ of measure $>(1-\epsilon) T$ such that for every $t\in \mathscr T$, each of the $r$
intervals $\big(t+(j-1)h), t+jh \big],\, j=1, 2, \ldots, r$ contains $<\frac32 m$ ordinates of zeros, at least one consecutive pair of which forms a gap of length
$\geq 16\pi/(15 \log T)$.
\end{proof}

 \section{An application} \label{Sec: 4}

In the theory of the Riemann zeta function, one often encounters integrals around rectangles 
passing through the critical strip  whose integrands contain  the functions  $\zeta'/\zeta(s)$ or $\xi'/\xi(s)$, where 
$\xi(s)=\frac12 s(s-1)\pi^{-\frac s2}\Gamma(\frac s2)\zeta(s)$. We saw in \S\ref{Sec Intro} that one can find a sequence of $T\to \infty$ so that $\min_{\g} |\g-T|\geq 2\pi c/\log T$ for some positive $c$. 
If  an edge of the rectangle contains a segment such as $[-1+iT, 2+iT]$, 
it then follows from 
the formula 
\be\label{z'/z}
\frac{\zeta'}{\zeta}(s) =\sum_{|t-\gamma|\leq 1} \frac{1}{s-\rho}  +  O(\log |t|)
\ee
for $-1\leq \sigma\leq 2$ (see Davenport~\cite[p. 99]{D}), that $\zeta'/\zeta(s)\ll (\log T)^2$ on this segment.   
This allows one to bound the contribution of this part of the contour to the overall integral. Equation \eqref{z'/z} also
holds for $\xi'/\xi(s)$, so we also have $\xi'/\xi(s)\ll (\log T)^2$ on this segment.

One may similarly ask for a bound for $\xi''/\xi'(s)$ on such a segment.  This question arises, for example in Farmer, Gonek, and Lee~\cite{FGL} 
(see (6.14)), where the bound
\[ \xi''/\xi'(\sigma+iT_j)\ll (\log T_j)^2 \]
(for $-1 \leq \sigma \leq 2$ and $T_j$ a sequence going to infinity) was stated without proper justification. 
The following theorem, an application of Theorem~\ref{thm: uncond} of the last section, gives a slightly weaker bound than the one claimed in~\cite{FGL} which is nevertheless sufficient for the argument there.

\begin{theorem} \label{thm: app}
Assume the Riemann Hypothesis.  If $C>2$, then there is a sequence $T_j\to \infty$ for which
\be\label{Bd to prove 1}
\frac{\xi''}{\xi'}(\s+iT_j)\ll  (\log T_j)^{C+1}
\ee
for  $-1\leq \sigma\leq 2$.  
\end{theorem}
\begin{proof} 
Observe that by the functional equation,  $\xi(s)=\xi(1-s)$, we have
\be\notag
\frac{\xi''}{\xi'}(s) =- \frac{\xi''}{\xi'}(1-s). 
\ee
Thus, it suffices to prove \eqref{Bd to prove 1}
for $\frac12\leq \sigma\leq 2$.
We do this by arguing from the easily obtained formula
\be\label{xi'/xi formula} 
\frac{\xi''}{\xi'}(s)= \frac{\xi'}{\xi}(s) +\frac{ ({\xi'}/{\xi})'(s)}{ {\xi'}/{\xi}(s)}.  
\ee

Now, if $s\neq \r$,
\be\label{xi'/xi}
\frac{\xi'}{\xi}(s) =B+ \sum_{\rho }  \Big(\frac{1}{s-\rho}   -\frac 1\rho \Big),
\ee
where $B$ is a constant. Hence,
\be\label{4}
\Big(\frac{\xi'}{\xi}\Big)'(s) = - \sum_{\rho }\frac{1}{(s-\rho)^2}  .
\ee
In particular, if the Riemann Hypothesis is true and $t\neq \gamma$, then
\be\notag
\frac{d}{dt} \Big(\frac{\xi'}{\xi}(\tfrac12+it) \Big)=   \sum_{\rho }\frac{-1}{(t-\gamma)^2} <0 .
\ee
Thus, between two consecutive ordinates, $i {\xi'}/{\xi}(\frac12+it)$   decreases
from infinity to minus infinity, and there is a unique point $\g^*$ in the interval such that $\xi'(\frac12+i\g^*)=0$. Clearly if a bound such as \eqref{Bd to prove 1} is to hold, $T$ cannot be too close to $\g^*$.

 By Theorem~\ref{thm: uncond} with 
 $r=3,\epsilon>0$, and  $h=2\pi m/\log T$ with $m$ and $T$ large enough, there exists a set $\mathscr T$ of measure $>(1-\epsilon) T$ in $(T, 2T]$ such that for every $t\in \mathscr T$, each of the three
intervals $(t, t+h], (t+h, t+2h], (t+2h, t+3h]$ contains $<\frac32 m$ ordinates of zeros, at least one consecutive pair of which forms a gap of length
$\geq16\pi/(15\log T)$. Let $\g_1, \g_2$ be such a pair in the first interval, $\g_3, \g_4$ a pair in the second, and $\g_5, \g_6$ a pair in the third.
The corresponding zeros will be $\r_j=\frac12+i\g_j$.
Let $\g_3^*$ denote the unique point in the middle interval $[\g_3, \g_4]$ such that $\xi'(\frac12+i\g_3^*)=0$. Without loss of generality, we assume that $\g_3^*-\g_3 \geq 8\pi/(15\log T)$, the alternative being that $\g_4-\g_3^* \geq 8\pi/(15\log T)$.

Now let $T_j  =\g_3 +(\log \g_3)^{-C}$ with $C>2$, and  $s=\s +iT_j$. As was mentioned above, we need only bound 
$\xi''/\xi'(\s+iT_j)$  when $\frac12\leq \s\leq 2$. We do this by splitting the interval into  two parts. 
 
\noindent{\bf Case 1.}   Assume $\frac12+ (\log \g_3)^{-C} < \s\leq 2$.

By \eqref{xi'/xi}, and since 
$B=\sum1/\rho$,
\be\label{xi'/xi lower 1}
\begin{split}
\Big|\frac{\xi'}{\xi}(\s+iT_j)\Big| \geq  \Re \frac{\xi'}{\xi}(\s+iT_j)
= & \sum_{\rho } \frac{\s-\frac12}{(\s-\frac12)^2 +(\g- T_j)^2}  
\geq   \frac{\s-\frac12}{(\s-\frac12)^2 +(\g_3- T_j)^2} \\
= &\frac{\s-\frac12}{(\s-1)^2 +(\log \g_3)^{-2C}}.
\end{split}
\ee
Also, by  \eqref{4},  
\be\notag
\begin{split}
\bigg(\frac{\xi'}{\xi}\bigg)^{'}(\s+iT_j) 
= -\sum_{|\g-T_j|\leq 1}   \frac{1}{(\s+iT_j-\rho)^2  }  
-\sum_{|\g-T_j|>1}   \frac{1}{(\s+iT_j-\rho)^2  }.
\end{split}
\ee
Since from \eqref{zero ct} we have $N(T+1)-N(T)\ll \log T$,
we see that
\be\notag%
\bigg(\frac{\xi'}{\xi}\bigg)^{'}(\s+iT_j) \ll  \frac{\log T_j}{(\s-\frac12)^2 }+ \log T_j \ll \frac{\log T_j}{(\s-\frac12)^2 }.
\ee 
It follows from this and \eqref{xi'/xi lower 1} that the second term on the right-hand side of \eqref{xi'/xi formula}  is
\be\notag%
\begin{split}
\frac{ ({\xi'}/{\xi})'(\s+iT_j)}{ {\xi'}/{\xi}(\s+iT_j)} 
\ll &\log T_j \Big(\frac{1}{\s-\frac12} +        \frac{1}{(\s-\frac12)^3(\log \g_3)^{2C}  }    \Big)\\
\ll & \log T_j (\log \g_3)^C \ll (\log T_j)^{C+1}.
\end{split}
\ee
To bound the first term on the right-hand side  in \eqref{xi'/xi formula}, note that 
\[
\frac{\xi'}{\xi}(s) = \frac{\zeta'}{\zeta}(s)+O(\log (|s|+2))
\]
(see (7) and (8) of Davenport~\cite{D}, p.80)
Thus, by \eqref{z'/z}, 
\be\label{xi'/xi 2} 
\frac{\xi'}{\xi}(\s+iT_j) =\sum_{|T_j-\gamma|\leq 1} \frac{1}{\s+iT_j-\rho}  +  O(\log  T_j).
\ee
Hence,
\be\notag
\begin{split}
 \frac{\xi'}{\xi}(\s+iT_j) 
\ll&  \frac{\log T_j}{\s-\frac12  }+ \log T_j 
\ll \frac{\log T_j}{\s-\frac12 }\ll (\log T_j)^{C+1}.
\end{split}
\ee
Combining these estimates in \eqref{xi'/xi formula}, we find that
\be\notag
\frac{\xi''}{\xi'}(\s+iT_j)\ll  \frac{\log T_j}{\s-\frac12  }
\ll ( \log T_j)^{C+1}
\ee
when  $\frac12+ (\log \g_3)^{-C} < \s\leq 2$.

\noindent{\bf Case 2.}   Assume $\frac12\leq  \s\leq\frac12+ (\log \g_3)^{-C} $. 

We first derive a lower bound for $\xi'/\xi(\s+iT_j)$.
By \eqref{xi'/xi 2} 
\be\label{xi'/xi 3}
\begin{split}
 \frac{\xi'}{\xi}(\s+iT_j) 
 =&\sum_{T_j-1\leq \gamma \leq \g_3} \frac{1}{\s+iT_j-\rho} +\sum_{\g_4\leq \g \leq T_j+1} \frac{1}{\s+iT_j-\rho} +  O(\log T_j)\\
 =& S_1 +S_2 +  O(\log T_j),
\end{split}
\ee
say.
In $S_2$, we have $\g-T_j\gg 1/\log T_j$ for each term.  Hence, as there are $O(\log T_j)$ terms,  $S_2\ll (\log T_j)^2$.
We next  write
\[
S_1=\sum_{T_j-1\leq \gamma \leq \g_1} \frac{1}{\s+iT_j-\rho} + \sum_{\g_1<\gamma \leq \g_3} \frac{1}{\s+iT_j-\rho} .
\]
In the first sum  we again have  $\g-T_j\gg 1/\log T_j$ for each term, and there are again $O(\log T_j)$ terms. Thus, this sum is $\ll(\log T_j)^2$. 
Combining our estimates so far with \eqref{xi'/xi 3}, we see that
\be\label{xi'/xi 4}
 \frac{\xi'}{\xi}(\s+iT_j) =  \sum_{\g_1<\gamma \leq \g_3} \frac{1}{\s+iT_j-\rho} +O((\log T_j)^2).
\ee
Now
\be\notag
\begin{split}
\Big|\sum_{\g_1<\gamma \leq \g_3} \frac{1}{\s+iT_j-\rho}\Big|\geq& \Big|\Im \sum_{\g_1<\gamma \leq \g_3} \frac{1}{\s+iT_j-\rho}\Big| 
= \sum_{\g_1<\gamma \leq \g_3} \frac{T_j-\g}{(\s-\frac12)^2+(T_j-\g)^2}\\
 \geq & \frac{T_j-\g_3}{(\s-\frac12)^2+(T_j-\g_3)^2} 
 = \frac{(\log \g_3)^{-C}}{(\log \g_3)^{-2C}+(\log \g_3)^{-2C}}\\
 =&\frac12 (\log \g_3)^{C} \gg (\log T_j)^C.
\end{split}
\ee
Thus, since $C>2$, we deduce from \eqref{xi'/xi 4} that
\be\label{xi'/xi 5}
 \frac{\xi'}{\xi}(\s+iT_j) \gg (\log T_j)^C.
\ee

To obtain an upper bound for $(\xi''/\xi')(\s+iT_j)$ from the formula in \eqref{xi'/xi formula}, 
we also require upper bounds for $(\xi'/\xi)(\s+iT_j)$ and $(\xi'/\xi)'(\s+iT_j)$. Appealing to  
 \eqref{xi'/xi 4}, we observe that there are $\ll m$ terms in the sum and that each is $\ll (\log T_j)^C$. Hence
 \be\label{xi'/xi upper}
  \frac{\xi'}{\xi}(\s+iT_j) \ll  (\log T_j)^{C}.
\ee 
To obtain an upper bound for $(\xi'/\xi)'(\s+iT_j)$, notice that by \eqref{4},
\be\notag
\begin{split}
\Big(\frac{\xi'}{\xi}\Big)^{'}(\s+iT_j) 
=-\sum_{|\g-T_j|\leq 1}   \frac{1}{(\s+iT_j-\rho)^2  }  -\sum_{|\g-T_j|>1}   \frac{1}{(\s+iT_j-\rho)^2  }.
\end{split}
\ee
The second sum is clearly $O( \log T_j)$. In the first sum there are $O(\log T_j)$ terms, each of which is of size at most $O((\log T_j)^{2C})$. 
Thus, altogether we have 
\be\label{der xi'/xi}
 \Big(\frac{\xi'}{\xi}\Big)^{'}(\s+iT_j) \ll (\log T_j)^{2C+1 }.
\ee
We now combine \eqref{xi'/xi 5}-\eqref{der xi'/xi} with \eqref{xi'/xi formula} to find that 
\[
\frac{\xi''}{\xi'}(\s+i T_j)\ll ( \log T_j)^{C+1}
\]
on the interval $\frac12\leq  \s\leq\frac12+ (\log \g_3)^{-C} $. 

Combining the  upper bounds we obtained in Case 1. and Case 2., we see that
\be\notag
\frac{\xi''}{\xi'}(\s+iT_j)\ll ( \log T_j)^{C+1}
\ee
for $\frac12\leq  \s \leq 2$. This completes the proof of the  theorem.
\end{proof}

\subsection*{Acknowledgements}
The authors would like to thank David Farmer, David de Laat, Brad Rodgers, Larry Rolen, Tim Trudgian, Zachary Tripp, Caroline Turnage-Butterbaugh, and Ian Wagner for helpful discussions and correspondence. In the course of writing this paper, the second author was partially supported by Purdue University start-up funding available to Trevor Wooley, summer funding from the NSF FRG grant DMS-1854398, and the AMS-Simons Travel Grant. 

\subsection*{Disclosure on the use of AI tools}
A large language model was used to find errors in the first version of this paper. The model found oversights in the proofs of Theorem~\ref{thm: uncond} and Theorem~\ref{thm: app} and also suggested easy fixes, which have been incorporated. LLM tools were also used in formatting the bibliography. 


\begin{thebibliography}{BM}

\bibitem{Anthropic} L.~Alp\"oge and R.~Furman, \emph{More than two thirds of the zeta zeros are simple and on the critical line}, 2026, arXiv:2608.13637.

\bibitem{ARTZ} C.~Ang\'on, E.~Ross, Z.~S.~Tang, and B.~Zhao, \emph{Unconditional bounds for two, three and four consecutive moderate gaps between zeros of the Riemann zeta function}, 2026, arXiv:2610.22927.

\bibitem{AR} J.~Arias de Reyna and B.~Rodgers, \emph{On convergence of points to limiting processes, with an application to zeta zeros}, Expo. Math. {\bf 42} (5) (2024), Paper No. 125588.

\bibitem{Baluyot} S.~A.~C. Baluyot, \emph{On the pair correlation conjecture and the alternative hypothesis,} J. Number Theory {\bf 169} (2016),  183-226.

\bibitem{BH87} E.~Bombieri and D.~A.~Hejhal, \emph{Sur les z\'eros des fonctions z\^{e}ta d'Epstein}, C. R. Acad. Sci. Paris S\'er. I Math. {\bf 304} (9) (1987), 213--217.

\bibitem{BH} E.~Bombieri and D.~A.~Hejhal, \emph{On the distribution of zeros of linear combinations of Euler products}, Duke Math. J. {\bf 80} (3) (1995), 821--862.

\bibitem{Cicek} F.~\c{C}i\c{c}ek, \emph{On the logarithm of the Riemann zeta-function near the nontrivial zeros}, Trans. Amer. Math. Soc. {\bf 374} (8) (2021), 5995--6037.

\bibitem{CG} F.~\c{C}i\c{c}ek and S.~M.~Gonek, \emph{The uniform distribution modulo one of certain subsequences of ordinates of zeros of the zeta function}, Math. Proc. Cambridge Philos. Soc. {\bf 176} (3) (2024), 593--608.

\bibitem{CGK} F.~\c{C}i\c{c}ek, S.~M.~Gonek, and S.~J.~Kirila, \emph{A central limit theorem for linear combinations of logarithms of Dirichlet $L$-functions sampled at the zeros of the zeta function}, Ramanujan J. {\bf 67} (2025), Paper No.~31.

\bibitem{CI}  B.~Conrey and H.~Iwaniec, \emph{Spacing of zeros of Hecke L-functions and the class number problem}, Acta Arith., {\bf 103} (3) (2002), 259–312.
%
\bibitem{CT-B} J.~B.~Conrey and C.~L.~Turnage-Butterbaugh, \emph{On r-gaps between zeros of the Riemann zeta-function}, Bull. Lond. Math. Soc., {\bf 50} (2) (2018), 349–356.

\bibitem{D} H.~Davenport, \emph{Multiplicative Number Theory}, Graduate Texts in Math. \textbf{74}, Springer-Verlag, New York, 1980.

\bibitem{dCM} J.~des Cloizeaux and M.~L.~Mehta, \textit{Asymptotic behavior of spacing distributions for the eigenvalues of random matrices}, J. Mathematical Phys. {\bf 14} (1973), 1648--1650.

\bibitem{F} A.~Fujii, \emph{On the distribution of the zeros of the Riemann Zeta function in short intervals}, Bull. Amer. Math. Soc. {\bf 81} (1975), 139--142. 

\bibitem{FGL} D.~W.~Farmer, S.~M.~Gonek, and Y.~Lee, \emph{Pair correlation of the zeros of the derivative of the Riemann $\xi$-function}, J. Lond. Math. Soc. (2), {\bf 90} (1) (2014), 241–269.

\bibitem{Hejhal} D.~A.~Hejhal, \emph{On the distribution of $\log|\zeta'(\frac{1}{2}+it)|$}, in \emph{Number theory, trace formulas and discrete groups} (Oslo, 1987), Academic Press, Boston, MA, 1989, 343--370.

\bibitem{Kirila} S.~J.~Kirila, \emph{Discrete moments and linear combinations of $L$-functions}, Ph.D. thesis, University of Rochester, 2018.

\bibitem{Lamzouri} Y.~Lamzouri, \emph{Isolated zeros of the Riemann zeta function and zeros of linear combinations of Dirichlet $L$-functions}, 2026, arXiv:2610.26841.

\bibitem{Lester} S.~J.~Lester, \emph{$a$-points of the Riemann zeta-function on the critical line}, Int. Math. Res. Not. IMRN (2015), no.~9, 2406--2436.

\bibitem{M} H.~L.~Montgomery. \emph{The pair correlation of zeros of the zeta function}, In Analytic number theory (Proc. Sympos. Pure Math., Vol. XXIV, St. Louis Univ., St. Louis, Mo., 1972), pages 181–193. 1973.

\bibitem{Me} M. L. Mehta, \textit{Random matrices}, Pure and Applied Mathematics, \textbf{142}, 3rd ed., Elsevier/Academic Press, Amsterdam, 2004.

\bibitem{O} A. M. Odlyzko, \emph{On the distribution of spacings between zeros of the zeta function}, Math. Comp. {\bf 48} (177) (1987), 273--308.

\bibitem{RV} B. Rodgers and H. S. Vallabhaneni, \emph{Autocorrelations of characteristic polynomials for the Alternative Circular Unitary Ensemble}, Glasg. Math. J. {\bf 66} (1) (2024), 51--64.

\bibitem{RS} Z. Rudnick and P. Sarnak, \emph{Zeros of principal L-functions and random matrix theory}, Duke Math. J. {\bf 81} (2) (1996), 269--322.

\bibitem{Selberg} A.~Selberg, \emph{Old and new conjectures and results about a class of Dirichlet series}, in \emph{Proceedings of the Amalfi Conference on Analytic Number Theory} (Maiori, 1989), Univ. Salerno, Salerno, 1992, 367--385.

\bibitem{T} E. C. Titchmarsh, \emph{The Theory of the Riemann Zeta-Function} 2nd ed., revised by  D.R. Heath-Brown (Oxford Science Publications, 1986)

\end{thebibliography}
\bibliographystyle{plain}

\end{document}